\documentclass[12pt]{amsart}
\usepackage{amssymb,enumerate}
\usepackage{amsmath}

\usepackage{tabstackengine}
\usepackage[colorlinks=true, pdfstartview=FitV, linkcolor=blue, citecolor=blue, urlcolor=blue]{hyperref}
\usepackage{bm}
\usepackage{fullpage}

\newtheorem{theorem}{Theorem}[section]
\newtheorem{lemma}[theorem]{Lemma}
\newtheorem{proposition}[theorem]{Proposition}

\numberwithin{equation}{section}

 \newcommand{\RR}{\mathbb{R}}
  \newcommand{\HH}{\mathcal{H}}

\begin{document}

\title{Noncentral limit theorem for the generalized Rosenblatt process}

\author[D. Bell \and D. Nualart]{Denis Bell \and David Nualart}
\address{David Nualart:
Department of Mathematics,
University of Kansas,
Lawrence, KS 66045, USA}
\email{nualart@ku.edu}
\address{Denis Bell: Department of Mathematics, University of North Florida, Jacksonville, FL 32224, USA}
\email{dbell@unf.edu}
\thanks{D. Nualart was supported by the NSF grant  DMS 1512891}



\begin{abstract}

We use techniques of Malliavin calculus to study the convergence in law of a family of generalized Rosenblatt processes $Z_\gamma$ with kernels defined by parameters $\gamma$ taking values  in a tetrahedral region $\Delta$ of $\RR^q$. We prove that, as $\gamma$ converges to a face of $\Delta$, the process $Z_\gamma$ converges to a compound Gaussian distribution with random  variance given by the square of a Rosenblatt process of one lower rank. The convergence in law is shown to be stable. This work generalizes a previous result of Bai and Taqqu, who proved the result in the case $q=2$ and without stability.

\end{abstract}

\maketitle

 \begin{center}
 {\sc \small This paper is dedicated to the memory of Salah Mohammed}
 \end{center}

\section{Introduction}

Let $W=\{W_x, x\in \RR\}$ be a two-sided Brownian motion on the real line.
The {\it generalized Rosenblatt process} is defined by 
\begin{equation}  \label{y1}
Z_{\gamma} (t)=\int_{\RR^q}  f_\gamma(x_1, \dots, x_q) dW_{x_1} \cdots dW_{x_q}, \quad t\ge 0,
\end{equation}
where 
$q\ge 2$, $\gamma=(\gamma_1, \dots, \gamma_q)$ and
\[
f_{\gamma}  (x)= A _\gamma  \int_0^t  (s-x_1)_+^{\gamma_1} (s-x_2)_+^{\gamma_2}  \cdots (s-x_q)_+^{\gamma_q} ds
\]
for $x=(x_1, \dots, x_q) \in \mathbb{R}^q$. The constant $A_\gamma$ is a normalizing constant, chosen so that $E[Z_\gamma (t)^2] =t^{\bar{\gamma}+1+\frac q2}$, where $\bar{\gamma} =\gamma_1 + \cdots +\gamma_q$.
For $f_\gamma$ to be in $L^2(\RR^q)$  it is necessary that the exponent $\gamma$ live in the region
\[
\Delta=\{ \gamma:   -1 < \gamma_i < -\frac 12, 1\le i\le q,   \gamma_1 +\cdots + \gamma_q>-\frac {q+1}2\}.  
\]

For $q=1$, this process reduces to the fractional Brownian motion $B_H$ with Hurst parameter $H=\gamma_1 +\frac 32\in (\frac 12,1)$.
When $q=2$, the process has been considered by Maejima and Tudor in \cite{MaTu}, and it generalizes the classical Rossenblatt process ($q=2$, $\gamma_1=\gamma_2$)  introduced  by Taqqu in \cite{Ta}.

In a recent work by Bai and Taqqu \cite{BaiTa2}, the authors study the convergence in law of this process when $q=2$ and  the parameter $\gamma=(\gamma_1, \gamma_2)$ converges to the boundary of the region $\Delta$.  In particular, when $\gamma_1 \rightarrow-\frac 12$ and $\gamma_2$ is fixed, the limit in distribution is $\eta B_{\gamma_2+\frac 32}(t)$, where $\eta$ is a standard normal Gaussian variable independent of the  fractional Brownian motion  $B_{\gamma_2+\frac 32}$. Two different proofs are given of this result, one  based on the method of moments and a second constructive proof based on a discretization argument.    

 The goal of this paper is to derive this result as an application of a general theorem  of convergence in law  of multiple stochastic integrals   to a mixture of Gaussian distributions (see Theorem  \ref{thm2}), which is of independent interest. This theorem  is proved using a  noncentral limit theorem for Skorohod integrals derived by Nourdin and Nualart in \cite{NoNu}.
  This allows us to extend Bai and Taqqu's result in two directions: We can deal with   a general Rosenblatt  in the $q$th Wiener chaos, and we can show that the convergence is stable.
  
  On the other hand, using a version of the Fourth Moment Theorem of Nualart and Peccati \cite{NuPe}, we show (see Theorem \ref{thm5})  that when $\gamma_1+ \cdots + \gamma_q \rightarrow -\frac {q+1} 2$ and $\gamma_i >-1+\epsilon$, $1\le i\le q$, for a fixed $\epsilon>0$, then the limit is a standard Brownian motion $B(t)$. For $q=2$ this was  also proved in \cite{BaiTa2}.

\section{Preliminaries} 
\subsubsection{Multiple stochastic integrals}
We denote by $W=\{W(x), x\in \RR\}$ a two-sided Brownian motion on the real line defined on some probability space $(\Omega ,\mathcal{F},P)$. Then we can define the Wiener integral $W(h)=\int_\RR h(x) dW_x$ for any function $h$ in the Hilbert space $\HH:= L^2(\RR)$, and $\{W(h), h\in \HH\}$ is an isonormal Gaussian process. We recall that this means that this is a centered Gaussian family with covariance  given by the scalar product  in $\HH$:
\[
E[W(h) W(g)] =\langle h,g \rangle_{\HH}.
\]
For every integer $q\geq 1$, consider the tensor product $\HH^{\otimes q} = L^2(\RR^q)$ and the symmetric tensor product, denoted by $\HH^{\odot q}$, formed by the symmetric functions in $L^2(\RR^q)$. For  any symmetric function $f\in  \HH^{\odot q}$ we denote by $I_q(f)$ the multiple  Wiener-It\^o stochastic integral of $f$ with respect to $W$, that  can be defined as an iterated It\^o integral:
\[
I_q(f) = \int_{\RR^q}  f(x_1, \dots, x_q) dW_{x_1} \cdots dW_{x_q}.
\]
Then the following isometry formula holds:
\[
E[I_q(f)^2] = q! \| f\|^2_{L^2(\RR^q)}.
\]
If $f\in  L^2(\RR^q)$, we put  $I_q(f)= I_q({\tilde f})$, where $\tilde{f}$ denotes the symmetrization of $f$, that is,
\[
\tilde{f}_{\gamma}(x_1, \dots, x_q) =\frac 1 {q!} \sum_{\sigma} f(x_{\sigma_1}, \dots, x_{\sigma_q}),
\]
where $\sigma$ runs over all the permutations of $\{1,\dots, q\}$.

Let $m,q\ge1$ two integers. Given a subset $I \subset \{1,\dots, q\}$ of cardinality $r=0, \dots, q\wedge m$,  a one-to-one mapping $\psi:I\rightarrow \{1, \dots, m\}$, and two functions $f \in L^2(\mathbb{R}^q)$  and $g\in L^2(\RR^m)$, we denote by  $f  \otimes_{I,\psi} g$ the element in $L^2(\RR^{q+m-2r})$ given by
\[
f  \otimes_{I,\psi} g= \int_{\mathbb{R}^r} f(x_1, \dots, x_q) g(y_1, \dots, y_m) \prod_{j=1}^r \delta(x_i-y_{\psi(i)}) dx_1 \cdots dx_r.
\]
That is, $f  \otimes_{I,\psi} g$ is the function in  $f,g \in L^2(\mathbb{R}^{2q-2r})$ obtained by contracting the variables $x_i$ from $f$ with the variables $y_{\psi(i)}$ from $g$.
When $q=m$ and $I=\{1,\dots, q\}$, we simply write $f \otimes_\psi g$.
 Then  the following product formula for multiple stochastic integrals holds. For any $f\in L^2(\RR^q)$ and $g\in L^2(\RR^m)$,
 \begin{equation}  \label{prod}
 I_q(f)  I_m(g)=\sum_{r=0}^{q\wedge m} \sum_{I,\psi} I_{q+m-2r}(f\otimes_{I,\psi} g),
 \end{equation}
 where the sum runs over all sets  $I \subset \{1,\dots, q\}$ of cardinality $r$ and one-to-one mappings $\psi:I\rightarrow \{1, \dots, m\}$.  Notice that when $f$ and $g$ are symmetric, this reduces to the well-known formula
 \[
 I_q(f)  I_m(g)=\sum_{r=0}^{q\wedge m} {q\choose r} {m\choose r} r!  I_{q+m-2r}(f\otimes_r g),
 \]
 where  $f\otimes_r g$ is the contraction of $r$ indices of $f$ and $g$. On the other hand, for any function $f\in \HH^{\otimes q}$, which is not necessarily symmetric, we have
 \[
 E[I_q(f)^2] = \sum _\psi f \otimes_\psi f,
 \]
 where  $\psi$ runs over all bijections of $\{1,\dots, q\}$.

Let $\{F_{n}\}$ be a sequence of random variables, all defined on the probability space $(\Omega, \mathcal{F}, P)$ and let $F$ be a random variable defined on some extended probability space $(\Omega', \mathcal{F}', P')$.  We say that $F_{n}$ {\it converges  stably} to $F$, if
\[
\underset{n \rightarrow \infty}{\rm lim}E\left[Ze^{i \lambda  F_{n} } \right] = E'\left[Ze^{i\lambda, F }\right]
\]
for every $\lambda \in \RR$ and every bounded $\mathcal{F}$--measurable random variable $Z$.

 \subsubsection{Elements of Malliavin calculus}

We introduce some basic elements of the Malliavin calculus with respect
to the  two-sided Brownian motion $W$. We refer the reader to Nualart \cite%
{Nu} for a more detailed presentation of these notions. Let $\mathcal{S}$
be the set of all smooth and cylindrical random variables of
the form
\begin{equation}
F=g\left( W(h_{1}),\ldots ,W(h_{n})\right) ,  \label{v3}
\end{equation}%
where $n\geq 1$, $g:\mathbb{R}^{n}\rightarrow \mathbb{R}$ is a infinitely
differentiable function with compact support, and $h _{i}\in \HH$.
The Malliavin derivative of $F$ with respect to $X$ is the element of $%
L^{2}(\Omega ;\HH)$ defined as
\begin{equation*}
DF\;=\;\sum_{i=1}^{n}\frac{\partial g}{\partial x_{i}}\left( W(h
_{1}),\ldots ,W(h_{n})\right) h_{i}.
\end{equation*}
By iteration, one can
define the $q$th derivative $D^{q}F$ for every $q\geq 2$, which is an element of $L^{2}(\Omega ;
\HH^{\odot q})$.

For $q\geq 1$ and $p\geq 1$, ${\mathbb{D}}^{q,p}$ denotes the closure of $%
\mathcal{S}$ with respect to the norm $\Vert \cdot \Vert_{\mathbb{D}^{q,p}}$, defined by
the relation
\begin{equation*}
\Vert F\Vert _{\mathbb{D}^{q,p}}^{p}\;=\;E\left[ |F|^{p}\right] +\sum_{i=1}^{q}E\left(
\Vert D^{i}F\Vert _{\HH^{\otimes i}}^{p}\right) .
\end{equation*}
If $V$ is a real separable Hilbert space, we denote by $\mathbb{D}^{q,p}(V)$ the corresponding Sobolev space of $V$-valued random variables.

We denote by $\delta $ the adjoint of the operator $D$, also called the 
divergence operator. The operator $\delta $ is an extension of the It\^o integral. It is also called the
Skorohod integral because in the case of the Brownian motion it coincides
with the anticipating stochastic integral introduced by Skorohod in \cite{Sk}. 
A random element $u\in L^{2}(\Omega ;\HH)$ belongs to the domain of $\delta $, denoted $\mathrm{Dom}\delta $, if and
only if it satisfies
\begin{equation*}
\big|E\big(\langle DF,u\rangle _{\HH}\big)\big|\leq c_{u}\,\sqrt{E(F^2)}
\end{equation*}%
for any $F\in \mathbb{D}^{1,2}$, where $c_{u}$ is a constant depending only
on $u$. If $u\in \mathrm{Dom}\delta $, then the random variable $\delta (u)$
is defined by the duality relationship 
\begin{equation}
E(F\delta (u))=E\big(\langle DF,u\rangle _{\HH}\big),  \label{ipp}
\end{equation}%
which holds for every $F\in {\mathbb{D}}^{1,2}$.  The operators $D$ and $\delta$ satisfy the following commutation relation:
\begin{equation} \label{comm}
D(\delta (u))= u+ \delta(Du),
\end{equation}
for any $u\in \mathbb{D}^{2,2}(\HH)$.

\section{Noncentral limit theorems for multiple stochastic integrals}
  
The following result has been proved by Nourdin and Nualart in \cite{NoNu}.

\begin{theorem} \label{thm1}
Consider a sequence of Skorohod integrals of the form $F_n=\delta(u_n)$, where $u_n\in \mathbb{D}^{2,2}(\HH)$. Suppose that the sequence $\{F_n, n\ge 1\}$ is bounded in $L^1(\Omega)$ and the following conditions hold:
\begin{itemize}
\item[(i)] $\langle u_n, h \rangle_{\mathcal{H}} $  converges to zero in $L^1(\Omega)$ for  all elements $h\in \mathcal{H}_0$, where $\mathcal{H}_0$ is a dense subset of $\mathcal{H}$.
\item[(ii)] $\langle u_n, DF_n \rangle_{\mathcal{H}}$ converges in $L^1(\Omega)$ to a nonnegative random variable $S^2$.
\end{itemize}
Then $F_n$ converges stably to a random variable with conditional Gaussian law $N(0,S^2)$ given $W$.
\end{theorem}

\medskip
On the other hand, from  Proposition 3.1 of the paper by Nourdin, Nualart and Peccati \cite{NoNuPe}, it follows that for any test function $\varphi \in \mathcal{C}^3$, we have
\begin{equation} \label{rate}
|E[\varphi(F_n)]  -E[ \varphi(S\eta)]|  \le  \frac 12 \|\varphi'' \|_\infty E[ | \langle u_n, DF_n \rangle_{\mathcal{H}} - S^2 |]
+  \frac 13\| \varphi '''\|_\infty E[| \langle u_n, DS^2 \rangle_{\mathcal{H}} |],
\end{equation}
assuming $S^2 \in \mathbb{D}^{1,2}$, and where $\eta$ is a $N(0,1)$-random variable independent of the process $W$.  This provides a rate of convergence in the previous theorem. Moreover, the in order to show the convergence  in law  $F_n \Rightarrow S\eta$, it suffices to check the following two conditions:
\begin{itemize}
\item[(i)] $\langle u_n, DF_n \rangle_{\mathcal{H}}  \rightarrow S^2 $  in $L^1(\Omega)$ as $n$ tends to infinity, and
\item[(ii)]  $\langle u_n, DS^2 \rangle_{\mathcal{H}} \rightarrow 0$  in $L^1(\Omega)$ as $n$ tends to infinity.
\end{itemize}

\medskip
Applying Theorem  \ref{thm1},  we derive the following noncentral limit theorem for a sequence of multiple stochastic integrals of order $q$.

\begin{theorem}  \label{thm2}
Fix $q\ge 2$. Let $F_n   $ be given by
\[
F_n= \int_{\mathbb{R}^q} f_n (x_1,x_2, \dots, x_q) dW_{x_1} dW_{x_2} \cdots dW_{x_q},
\]
where  $f_n \in L^2(\mathbb{R}^q)$
Assume that:
\begin{itemize}
\item[(i)]   For all elements $h\in \mathcal{H}_0$, where $\mathcal{H}_0$ is a dense subset of $\mathcal{H}$, we have
\[
\int_{\RR} h(\xi) f_n (\xi, \cdot) d\xi
\]
converges to zero in $\HH^{\otimes(q-1)}$.
 \item[(ii)]   For any subset $I\subset \{1,\dots, q\}$ of cardinality $r=1, \dots , q$   and any one-to-one mapping $\psi:I\rightarrow \{1,\dots, q\}$ such that
 $1\in I$ and $\psi(1) \not =1$, 
 \[
f_n \otimes_{I,\psi} f_n
 \]
converges to zero in $\mathcal{H}^{\otimes(2q-2r)}$.
 
\item[(iii)]  There exists  an  element $g\in  L^2(\mathbb{R}^{q-1})$ with variables $g(x_{1}, \dots, x_q)$, such that for any subset $I \subset \{2,\dots, q\}$ of cardinality $r=0, \dots , q-1$  and any one-to-one mapping $\psi:I\rightarrow \{2,\dots, q\}$  
\[
\lim_{n \rightarrow \infty}  \int_{\mathbb{R}} [f_n(\xi, \cdot) \otimes_{I,\psi} f_n(\xi, \cdot)] d\xi = g \otimes _{I , \psi} g,
\]
where the convergence holds in  $\mathcal{H}^{\otimes(2q-2r-2)}$.
\end{itemize}
Then $F_n$ converges stably to a random variable with  conditional Gaussian law $N(0,S^2)$ given $W$, where
\[
S^2 =  \left( I_{q-1}(g) \right)^2.
\]
\end{theorem}

\begin{proof}
We can write $F_n=\delta(u_n)$ where  $u_n(\xi)=I_{q-1}(f_n(\xi, \cdot))$.   Then, we claim that $F_n$ and $u_n$ satisfy the conditions of Theorem  \ref{thm1}.
Notice tat $u_n\in \mathbb{D}^{2,2}(\HH)$ because $u_n$ is a multiple stochastic integral. 
To show condition (i) of Theorem \ref{thm1}, fix $h \in \HH$. Then,
\begin{eqnarray*}
E[ \langle u_n, h \rangle_\HH^2]&=& E\left[ \left| I_{q-1} \left( \int_{\RR} h(\xi) f_n(\xi, \cdot) d\xi \right) \right|^2 \right]\\
&\le & q!  \left \| \int_{\RR} h(\xi) f_n(\xi, \cdot) d\xi \right\|^2_{\HH^{\otimes (q-1)}},
\end{eqnarray*}
which converges to zero by condition (i).

It remains to check condition (ii) of Theorem \ref{thm1}.
Let us  first  compute the inner product   $ \langle u_n, DF_n \rangle_{\mathcal{H}}$.   Recall that $F_n=\delta (u_n)$, where $u_n(\xi)= I_{q-1} (f_n(\xi, \cdot))$.  Using the commutation relation (\ref{comm}), we can write
\begin{eqnarray*}
D_\xi F_n &=& u_n(\xi)+ \delta(D_\xi u_n)\\
&=& u_n(\xi)+  \sum_{i=2}^q \int_{\mathbb{R}^{q-1}}  f_n(x_1,  x_2 ,\dots, x_{i-1} ,\xi, x_{i+1}, \cdots, x_{q-1}) dW_{x_1} \cdots dW_{x_{q-1}}\\
&=:& u_n(\xi) + \sum_{i=2}^q G_{n,i}(\xi).
\end{eqnarray*}
We claim that  $\langle u_n, G_{n,i} \rangle_{\mathcal{H}}$ converges to zero in $L^2(\Omega)$. Indeed,
\begin{eqnarray*}
\langle u_n, G_{n,i} \rangle_{\mathcal{H}}&=& \int_{\mathbb{R}}    \left(
 \int_{\mathbb{R}^{q-1}} f_n (\xi,x_1, \dots, x_{q-1}) dW_{x_1} \cdots dW_{x_{q-1}}  \right) \\
 &&\times \left(  \int_{\mathbb{R}^{q-1}}  f_n(x_1,  x_2 ,\dots, x_{i-1} ,\xi, x_{i+1}, \cdots, x_{q-1}) dW_{x_1} \cdots dW_{x_{q-1}}\right) d\xi
\end{eqnarray*}
As a consequence, using the product formula for multiple stochastic integrals (see (\ref{prod})), we can write
\[
\langle u_n, G_{n,i} \rangle_{\mathcal{H}}=  
\sum_{r=1}^{q}  \sum_{I, \psi}  
 I_{2q-2r-2}   (f_n\otimes_{I,\psi} f_n),
\]
where the sum is over all sets $I \subset \{1,\dots, q\}$ of cardinality $r$ and one-to-one mappings $\psi: I\rightarrow \{1,\dots ,q\}$ such that  $i\in I$ and $\psi(1)=i$. Because $i\not =1$, by condition (i) this term
 converges to zero in $L^2(\Omega)$.
Finally,  taking into account that
\[
\langle u_n, DF_n \rangle_{\mathcal{H}}=  \|u_n\|^2_{\mathcal{H}} +\sum_{i=2}^q\langle u_n, G_{n,i} \rangle_{\mathcal{H}},
\]
it suffices to consider the convergence of  $ \|u_n\|^2_{\mathcal{H}} $.  For this term, we have
\begin{eqnarray*}
 \|u_n\|^2_{\mathcal{H}} &=& \int_{\mathbb{R}} I_{q-1} (f_n(\xi, \cdot))^2 d\xi  \\
&=& \sum_{r=0}^{q-1}  \sum_{I, \psi} I_{2q-2r-2}\left( \int_{\RR} (f_n(\xi, \cdot) \otimes_{I,\psi} f_n(\xi,\cdot)) d\xi\right),
\end{eqnarray*}
where the sum is over all sets $I \subset \{2,\dots, q\}$ of cardinality $r$ and one-to-one mappings $\psi: I\rightarrow \{2, \dots,q\}$. 
By our hypothesis  (ii), this sum converges in $L^2(\Omega)$ to
\[
  \sum_{r=0}^{q-1}     \sum_{I,\psi} I_{2q-2r-2}  ( g \otimes_{ I, \psi} g)  =  (I_{q-1}(g))^2,
 \] 
where the sum runs over all sets  $I \subset \{2,\dots, q\}$ of cardinality $r=0, \dots , q-1$  and any one-to-one mappings $\psi:I\rightarrow \{2,\dots, q\}$.  
  This completes the proof.
 \end{proof}
 
It will have been noted that the proof of Theorem 3.2 depends crucially upon expressing $F_n$ as the Skorohod integral of a multiple Wiener integral of rank $q-1$, i.e.  choosing a kernel $f_n$ such that $u_n(\xi)= I_{q-1}(f_n(\xi))$. Obviously, the choice of $f_n$ is not unique, e.g. one could equally well choose
$ f_n(\xi) = f_n(x_1,x_2,\dots,, x_{i-1},\xi, x_{i+1},\dots,x_{q-1})$, for any $2\le i\le q$. However, any such choice will lead to the term     
$$
\Big |\Big |\int_{\RR^{q-1}} f_n(x_1,x_2,\dots,, x_{i-1},\xi, x_{i+1},\dots,x_{q-1}) dW_{x_1} \cdots dW_{x_{q-1}}\Big |\Big |_\mathcal{H}^2.
$$
in the computation of $\langle u_n, DF_n \rangle_{\mathcal{H}}$. This term evidently does not converge in $L^2(\Omega)$ since it can be shown that  its $L^2$ norm converges to a non-zero limit, while the integrand \break $f_n(x_1,x_2,\dots,, x_{i-1},\xi, x_{i+1},\dots,x_{q-1})$ converges {\it pointwise to zero} outside of  the diagonal  in $\RR^{q-1}$. Thus the choice $i=1$ is the only one that will work in the argument.

\

 The case where the limit is Gaussian is not included in Theorem \ref{thm2}. We state this convergence in the next theorem, whose proof would be similar to that of Theorem \ref{thm2}. Notice that  Theorem \ref{thm4} below is just the Fourth Moment Theorem proved by Nualart and Peccati in \cite{NuPe} (see the reference \cite{np-book} for extensions and applications of this result).  In the version below of the Fourth Moment Theorem we do not require the kernels to be symmetric and we add condition (i) which ensures the stability of the convergence.
  
 \begin{theorem}  \label{thm4}
 Fix $q\ge 2$. Let $F_n $ be given by
\[
F_n= \int_{\mathbb{R}^q} f_n (x_1,x_2, \dots, x_q) dW_{x_1} dW_{x_2} \cdots dW_{x_q},
\]
where $f_n \in L^2(\mathbb{R}^q)$.
 Suppose that condition (i) in Theorem \ref{thm1} holds, and moreover,
 \begin{itemize}
 \item[(ii)] For any subset $I\subset \{1, \dots, q\}$ of cardinality $r=1,\dots, q-1$ and any one-to-one mapping $\psi: I\rightarrow \{1,\dots, q\}$, we have
 \[
 f_n \otimes_{I,\psi}  f_n \rightarrow 0
\]
in  $\HH^{\otimes (2q-2r)}$.
\item[(iii)] $\lim_{n\rightarrow \infty} E[F_n^2]=  \sigma^2  $.
\end{itemize}
Then, $F_n$ converges stably to a random variable with  Gaussian law $N(0,\sigma^2)$, independent of $W$.
  \end{theorem}
 
 \section{Generalized Rosenblatt process}

We are interested in the asymptotic behavior of  the generalized Rosenblatt process $Z_\gamma(t)$ defined in (\ref{y1}), when the parameter $\gamma$ converges to the boundary of the region $\Delta$. Consider first the case when one of the parameters (for simplicity we choose the first one) converges to $-\frac 12$.
 
 We will make use of the following  technical lemmas. The first lemma was proved by Bai and Taqqu in \cite[Lemma 3.2]{BaiTa1}.

\begin{lemma}  \label{lem1}
Suppose $-1<\gamma_1,\gamma_2<-1/2$ and $s_1, s_2>0$. Then 
\begin{eqnarray*}
\int_{-\infty}^\infty (s_1-x)_+^{\gamma_1}(s_2-x)^{\gamma_2}_+dx
&=& (s_2-s_1)_+^{1+\gamma_1+\gamma_2}B(1+\gamma_1, -1-\gamma_1-\gamma_2) \\
&&+
(s_1-s_2)_+^{1+\gamma_1+\gamma_2}B(1+\gamma_2, -1-\gamma_1-\gamma_2).
\end{eqnarray*}
\end{lemma}

The second lemma concerns the asymptotic behavior of the Beta function (see \cite[Lemma 3.8]{BaiTa2}).

\begin{lemma} \label{lem2}
As $\alpha \rightarrow 0$, we have
\[
\alpha B(\alpha, \beta) \rightarrow 1,
\]
uniformly in $\beta \in [b_0, b_1]$, where $0<b_0<b_1 <\infty$.
\end{lemma}

In the next lemma, we compute the explicit value of the constant $A_\gamma$.   
\begin{lemma} \label{lem3}
The constant $A_\gamma$ is given by
\[
A_\gamma^2=  \frac {  (2|\gamma| +q+1)(2|\gamma|+q+2)} {2   \sum_\sigma   \prod_{j=1}^q B(\gamma_j+1, -\gamma_j-\gamma_{\sigma_j}-1)},
\]
where the sum runs over all permutations $\sigma$ of $\{1,\dots, n\}$ and we recall that $|\gamma| =\sum_{j=1}^q \gamma_j$.
\end{lemma}

\begin{proof}
By a scaling argument, we can take $t=1$.
We have
\[
A_\gamma^2=q! \| \tilde{f}_{\gamma} \|_{L^2(\mathbb{R}^q)}^{-2},
\]
where  $ \tilde{f}_{\gamma} $ denotes the symmetrization of $f_\gamma$. Then
\begin{eqnarray*}
\tilde{f}_{\gamma}(x_1, \dots, x_q) ^2&= &\frac 1{(q!)^2} \sum_{\sigma,\tau}  \int_{[0,1]^2} (s_1-x_{\sigma_1})_+^{\gamma_1} \cdots
 (s_1-x_{\sigma_q})_+^{\gamma_q } \\
 &&\times (s_2-x_{\tau_1})_+^{\gamma_1} \cdots  (s_2-x_{\tau_q})_+^{\gamma_q} ds_1 ds_2.
 \end{eqnarray*}
As a consequence,
\begin{eqnarray*}
\|\tilde{f}_{\gamma} \|_{L^2(\mathbb{R}^q)} ^2&=& \frac 1{q!}  \sum_\sigma \int_{\mathbb{R}^q} \int_{[0,1]^2}   (s_1-x_1)_+^{\gamma_1} \cdots
 (s_1-x_{q})_+^{\gamma_q } \\
 &&\times (s_2-x_{\sigma_1})_+^{\gamma_1} \cdots  (s_2-x_{\sigma_q})_+^{\gamma_q} ds_1 ds_2 dx_1 \cdots dx_q\\
 &=&\frac 1{q!}  \sum_\sigma   \int_{[0,1]^2}  \left(  \prod_{j=1}^q \int_{\mathbb{R}} (s_1-x)_+^{\gamma_j} (s_2-x)_+^{\gamma_{\sigma_j}} dx \right) ds_1ds_2.
\end{eqnarray*}
By Lemma  \ref{lem1}, we have
\begin{eqnarray*}
\int_{\mathbb{R}} (s_1-x)_+^{\gamma_j} (s_2-x)_+^{\gamma_{\sigma_j}} dx
&=& (s_2-s_1)_+^{\gamma_j +\gamma_{\sigma_j}+1} B(\gamma_j+1, -\gamma_j-\gamma_{\sigma_j}-1) \\
&&+ (s_1-s_2)_+^{\gamma_j +\gamma_{\sigma_j}+1} B(\gamma_{\sigma_j}+1, -\gamma_j-\gamma_{\sigma_j}-1).
\end{eqnarray*}
Substituting this formula in the above expression for $ \|\tilde{f}_{\gamma} \|_{L^2(\mathbb{R}^q)} ^2$, we obtain
\begin{eqnarray*}
\|\tilde{f}_{\gamma} \|_{L^2(\mathbb{R}^q)} ^2
&=&\frac 1{q!}  \sum_\sigma   \int_{[0,1]^2}   \Bigg[ (s_2-s_1)_+^{2|\gamma|+q}  \prod_{j=1}^q B(\gamma_j+1, -\gamma_j-\gamma_{\sigma_j}-1) \\
&&+ (s_1-s_2)_+^{2|\gamma| +q}  \prod_{j=1}^q B(\gamma_{\sigma_j}+1, -\gamma_j-\gamma_{\sigma_j}-1) \Bigg] ds_1 ds_2 \\
&=& \frac {2   \sum_\sigma   \prod_{j=1}^q B(\gamma_j+1, -\gamma_j-\gamma_{\sigma_j}-1)}
{ q! (2|\gamma| +q+1)(2|\gamma|+q+2)},
\end{eqnarray*}
which completes the proof of the lemma.
\end{proof}

The following is the main result of this paper.

\begin{theorem} \label{thm4}
As  $\gamma_1$ converges to $-\frac 12$, the random variable  $Z_{\gamma} (t)$ converges stably to a  random variable whose distribution given $W$ is  Gaussian  with zero mean and  variance
$Z^2_{\gamma_{2}, \dots, \gamma_q}(t)$.
\end{theorem}

\begin{proof}
  To simplify, by a scaling argument we can assume that $t=1$.
The asymptotic behavior of the constant $A_\gamma$, when $\gamma_1 \rightarrow -\frac 12$ is obtained from Lemma \ref{lem3}, taking into account the asymptotic behavior of the Beta function given by Lemma \ref{lem2}:
\begin{equation} \label{eq1}
\lim_{\gamma_1 \rightarrow -\frac 12} \frac {A_\gamma^2}{ -1-2\gamma_1 } = \frac {
(2\sum_{j=2}^q \gamma_j +q)(2 \sum_{j=2}^q \gamma_j+q+1)}{ 2 \sum_\sigma \prod_{j=2}^q B(\gamma_j+1, -\gamma_j -\gamma_{\sigma_j}-1)}
=A^2_{\gamma_2, \dots, \gamma_q},
\end{equation}
where in the denominator of the second expression, $\sigma$ runs over all permutations of $\{2,\dots, q\}$.

The proof will be done in three steps.

\medskip
\noindent
{\it Step 1.} Let us show condition (i) of Theorem \ref{thm2}. We can take $h= \mathbf{1}_{[a,b]}$. Then,
\[
\int_a^b f_\gamma(\xi, \cdot) d\xi=A _\gamma  \int_0^1   \left( \int_a^b(s-\xi)_+^{\gamma_1}  d\xi \right)(s-x_2)_+^{\gamma_2}  \cdots (s-x_q)_+^{\gamma_q} ds.
\]
The term  $\int_a^b(s-\xi)_+^{\gamma_1}  d\xi $ is uniformly bounded as $\gamma_1\rightarrow -\frac 12$ and $A_\gamma$ converges to zero. Therefore,  the above expression converges to zero in $\HH^{\otimes(q-1)}$.

\medskip
\noindent
{\it Step 2.} Now we show condition (ii) of Theorem \ref{thm2}.
Fix a subset $I\subset \{1,\dots, q\}$ of cardinality $r=1, \dots , q$   and any one-to-one mapping $\psi:I\rightarrow \{1,\dots, q\}$ such that
 $1\in I$ and $\psi(1) \not=1$. Set $J=\psi(I)$.
Let us compute
   \begin{eqnarray*}
(f_n \otimes_{I,\psi} f_n)(x,y)&=&A_\gamma^2   \notag
  \int_{\mathbb{R}^r} \int_{[0,1]^2}   \prod_{i\in I} (s_1-\xi_i)_+^{\gamma_i} \prod_{j\in I^c} (s_1-x_j)_+^{\gamma_j}
  \\
  &&\times  \prod_{i\in I} (s_2-\xi_i)_+^{\gamma_{\psi(i)}} \prod_{k\in J^c} (s_2-y_k)_+^{\gamma_k}  \label{eq2}
  ds_2 ds_2  d\xi,
  \end{eqnarray*}
 where $x=(x_j)_{j\in I^c}$, $y=(y_k)_{k\in J^c}$, $\xi=(\xi_i)_{i\in I}$ and $d\xi= \prod_{i\in I} d\xi_i$.
Using Lemma \ref{lem1}, we obtain
    \begin{eqnarray*}
(f_n \otimes_{I,\psi} f_n)(x,y)&=&A_\gamma^2   
  \int_{[0,1]^2}   \Big(  \prod_{i\in I} (s_2-s_1)_+^{\gamma_i+\gamma_{\psi(i)} +1}
  B(\gamma_i +1, -\gamma_i -\gamma_{\psi(i)} -1) \\
  && + 
   \prod_{i\in I} (s_1-s_2)_+^{\gamma_i+\gamma_{\psi(i)} +1}
  B(\gamma_{\psi(i)}+1, -\gamma_i -\gamma_{\psi(i)} -1) \Big) \\
&& \times   \prod_{j\in I^c} (s_1-x_j)_+^{\gamma_j}
       \prod_{k\in J^c} (s_2-y_k)_+^{\gamma_k}  \label{eq2}
  ds_2 ds_2,
  \end{eqnarray*}
  Set
      \begin{eqnarray*}
  \Phi_1(s_1, s_2)&=&\prod_{i\in I} (s_2-s_1)_+^{\gamma_i+\gamma_{\psi(i)} +1}
  B(\gamma_i +1, -\gamma_i -\gamma_{\psi(i)} -1)  \\
  &&
  + \prod_{i\in I} (s_1-s_2)_+^{\gamma_i+\gamma_{\psi(i)} +1}
  B(\gamma_{\psi(i)}+1, -\gamma_i -\gamma_{\psi(i)} -1).
    \end{eqnarray*}
  With this notation, we can write
  \[
\|  f_n \otimes_{I,\psi} f_n\|_{L^2(\RR^{2q-2r})}^2= A_\gamma^4 \int_{[0,1]^4}
  \Phi_1(s_1, s_2)  \Phi_1(s_3, s_4)  \Phi_2(s_1, s_3) \Phi_3( s_2, s_4) ds_1 ds_2 ds_3 ds_4,
\]
where
\[
  \Phi_2(s_1, s_3) = \prod_{j\in I^c} |s_3-s_1|^{2\gamma_j +1}
  B(\gamma_j +1, -2\gamma_j -1)   
  \]
  and
\[
  \Phi_3(s_2, s_4) = \prod_{k\in J^c} |s_4-s_2|^{2\gamma_k +1}
  B(\gamma_k+1, -2\gamma_k -1).
  \]
We know that $A^4_\gamma (-1-2\gamma_1)^{-2}$ converges to a finite limit. On the other hand,  $ \Phi_1$ converges to a finite limit because $1\in I$ but $\psi(1) \not=1$. Also,
$\Phi_2$ converges to  a finite sum because $1\not \in I^c$ and $\Phi_3$ diverges as $(-1-2\gamma_1)^{-1}$ as $\gamma_1\rightarrow -\frac  12$.  Therefore,
\[
\lim_{\gamma_1 \rightarrow -\frac  12}     (-1-2\gamma_1)^2 \int_{[0,1]^4} 
  \Phi_1(s_1, s_2)  \Phi_1(s_3, s_4)  \Phi_2(s_1, s_3) \Phi_3( s_2, s_4) ds_1 ds_2 ds_3 ds_4 =0
  \]

\medskip
\noindent
{\it Step 3.}  It remains to show condition (iii).
Define
\[
g(x_2, \dots, x_q)=A_{\gamma_2,\dots,\gamma_q} \int_0^1 (s-x_2)_+^{\gamma_2} \cdots (s-x_{q})_+^{\gamma_q} ds= f_{\gamma_2,\dots,\gamma_q}(x_2, \dots, x_q).
\]
Fix  $r=0,\dots, q-1$,  a set $I\subset\{2,\dots, q\}$ of cardinality $r$ and a one-to-one mapping $\psi: =I \rightarrow \{2,\dots, q\}$. 
Set $J=\psi(I)$. We also write $\bar{I}= I\cup \{1\}$ and $\bar{\psi}$ is the extension of  $\psi$ to $\bar{I}$ such that $\bar{\psi} (1)=1$.
We claim that 
\[
  f_\gamma\otimes_{\bar{1},  \bar{\psi} } f_\gamma
  \]
  converges in  $L^2(\mathbb{R}^{2q-2r-2})$ to $g\otimes_{I, \psi} g$.
  We have
  \begin{eqnarray}
  ( f_\gamma\otimes_{ \bar{I},  \bar{\psi} } f_\gamma)(x,y) &=& A_\gamma^2  \notag
  \int_{\mathbb{R}^r} \int_{[0,1]^2} (s_1 -\xi_1)_+^{\gamma_1}  \prod_{i\in I} (s_1-\xi_i)_+^{\gamma_i} \prod_{j\in I^c, j\not=1} (s_1-x_j)_+^{\gamma_j}
  \\
  &&\times (s_2 -\xi_1)_+^{\gamma_1}  \prod_{i\in I} (s_2-\xi_i)_+^{\gamma_{\psi(i)}} \prod_{k\in J^c, k\not=1} (s_2-y_k)_+^{\gamma_k}  \label{eq2}
  ds_2 ds_2 d\xi_1 d\xi,
  \end{eqnarray}
  where $x=(x_j)_{ j\in  I^c}$, $y=(y_k)_{k\in J^c}$ and $\xi= (\xi_i)_{i\in I}$.
  By Lemma  \ref{lem1}, we have
  \[
  \int_{\mathbb{R}}  (s_1 -\xi_1)_+^{\gamma_1}(s_2 -\xi_1)^{\gamma_1} d\xi_1=|s_2-s_1|^{2\gamma_1+1}
  B(\gamma_1+1, -2\gamma_1-1).
  \]
  Therefore,
    \begin{eqnarray*}
  ( f_\gamma\otimes_{ \bar{I},  \bar{\psi} } f_\gamma)(x,y) &=& A_\gamma^2  
  \int_{[0,1]^2} |s_2-s_1|^{2\gamma_1+1}
  B(\gamma_1+1, -2\gamma_1-1)    \\
  && \times  \Big(\prod_{i\in I} (s_2-s_1)_+^{\gamma_i + \gamma_{\psi(i)}}  B(\gamma_i+1, -\gamma_i -\gamma_{\psi(i)} -1) \\
  && +
  \prod_{i\in I} (s_1-s_2)_+^{\gamma_i + \gamma_{\psi(i)}}  B(\gamma_{\psi(i)}+1, -\gamma_i -\gamma_{\psi(i)} -1) \Big)\\
 &&\times     \prod_{j\in I^c, j\not=1} (s_1-x_j)_+^{\gamma_j}
  \prod_{k\in J^c, k\not=1} (s_2-y_k)_+^{\gamma_k}  \label{eq2}
  ds_2 ds_2 ,
  \end{eqnarray*}
   It suffices to show that the following quantities converge to  $ \| g\otimes_{I, \psi} g\|^2_{L^2(\mathbb{R}^{2q-2r-2})}$ as $\gamma_1 \rightarrow-\frac 12$:
  \begin{equation} \label{e1}
  \|f_\gamma\otimes_{ \bar{I},  \bar{\psi} } f_\gamma\|^2_{L^2(\mathbb{R}^{2q-2r-2})},
 \end{equation}
  and
   \begin{equation} \label{e3}
  \langle  f_\gamma\otimes_{ \bar{I},  \bar{\psi} } f_\gamma,  g\otimes_{I, \psi} g\rangle_{L^2(\mathbb{R}^{2q-2r-2})}.
   \end{equation}
   We will consider only the convergence of (\ref{e1}), and that of (\ref{e3}) is proved in the same way.
 As before, set
      \begin{eqnarray*}
  \Phi_1(s_1, s_2)&=&\prod_{i\in I} (s_2-s_1)_+^{\gamma_i+\gamma_{\psi(i)} +1}
  B(\gamma_i +1, -\gamma_i -\gamma_{\psi(i)} -1)  \\
  &&
  + \prod_{i\in I} (s_1-s_2)_+^{\gamma_i+\gamma_{\psi(i)} +1}
  B(\gamma_{\psi(i)}+1, -\gamma_i -\gamma_{\psi(i)} -1).
    \end{eqnarray*}
  With this notation, we can write
  \begin{eqnarray*}
\|  f_n \otimes_{\bar{I},\bar{\psi}} f_n\|_{L^2(\RR^{2q-2r-2})}^2 &=& A_\gamma^4    B(\gamma_1+1, -2\gamma_1-1)^2
  \int_{[0,1]^4}  |s_2-s_1|^{2\gamma_1+1}|s_3-s_4|^{2\gamma_1+1} \\
&&\times 
  \Phi_1(s_1, s_2)  \Phi_1(s_3, s_4)  \Phi_2(s_1, s_3) \Phi_3( s_2, s_4) ds_1 ds_2 ds_3 ds_4,
\end{eqnarray*}
where
\[
  \Phi_2(s_1, s_3) = \prod_{j\in I^c, j\not =1} |s_3-s_1|^{2\gamma_j +1}
  B(\gamma_j +1, -2\gamma_j -1)   
  \]
  and
\[
  \Phi_3(s_2, s_4) = \prod_{k\in J^c, k\not=1} |s_4-s_2|^{2\gamma_k +1}
  B(\gamma_k+1, -2\gamma_k -1).
  \]
As $\gamma_1\rightarrow -\frac 12$, by the monotone convergence theorem, we obtain
\begin{eqnarray*}
\lim_{\gamma_1\rightarrow -\frac 12}&=& A^2_{\gamma_2, \dots, \gamma_q}  \int_{[0,1]^4}   \Phi_1(s_1, s_2)  \Phi_1(s_3, s_4)  \Phi_2(s_1, s_3) \Phi_3( s_2, s_4) ds_1 ds_2 ds_3 ds_4\\
&=&  \| g\otimes_{I, \psi} g\|^2_{L^2(\mathbb{R}^{2q-2r-2})}.
\end{eqnarray*}
This completes the proof.
\end{proof}

When   $\gamma$ converges to the boundary of $\Delta$ defined by $\bar{\gamma} =\frac{q+1}2$, we obtain the following result, that generalizes Theorem 2.1 in \cite{BaiTa2}.

\begin{theorem}  \label{thm5}
Suppose that $\gamma_1 +\cdots + \gamma_q \rightarrow -\frac {q+1}2$ with $\gamma_i >-1+\epsilon$, $1\le i \le q$, for arbitrarily fixed $\epsilon>0$. Then, $Z_\gamma(t)$ converges in law to   $B(t)$, where $B$ is a Brownian motion independent of $W$.
\end{theorem}

\begin{proof}
Condition (i) of Theorem \ref{thm1} is easy to show. Then, it suffices to check  condition (ii) of Theorem \ref{thm4}. Fix a  subset $I\subset \{1, \dots, q\}$ of cardinality $r=1,\dots, q-1$ and a one-to-one mapping $\psi: I\rightarrow \{1,\dots, q\}$. We have
 \[
 \|f_n \otimes_{I,\psi}  f_n \|^2_{\HH^{\otimes(2q-2r)}}=  A_\gamma^4 \int_{[0,1]^4}
  \Phi_1(s_1, s_2)  \Phi_1(s_3, s_4)  \Phi_2(s_1, s_3) \Phi_3( s_2, s_4) ds_1 ds_2 ds_3 ds_4,
\]
where
     \begin{eqnarray*}
  \Phi_1(s_1, s_2)&=&(s_2-s_1)_+^{\sum_{i\in I}(\gamma_i+\gamma_{\psi(i)} )+r}
   \prod_{i\in I}B(\gamma_i +1, -\gamma_i -\gamma_{\psi(i)} -1)  \\
  &&
  +   (s_1-s_2)_+^{\sum_{i\in I} (\gamma_i+\gamma_{\psi(i)}) +r}
\prod_{i\in I}  B(\gamma_{\psi(i)}+1, -\gamma_i -\gamma_{\psi(i)} -1),
    \end{eqnarray*}
    \[
  \Phi_2(s_1, s_3) =  |s_3-s_1|^{2\sum_{j\in I^c}\gamma_j +q-r}
  \prod_{j\in I^c} B(\gamma_j +1, -2\gamma_j -1)   
  \]
  and
\[
  \Phi_3(s_2, s_4) =  |s_4-s_2|^{2\sum_{k\in J^c}\gamma_k +q-r}
  \prod_{k\in J^c} B(\gamma_k+1, -2\gamma_k -1).
  \]
  All the products of Beta functions are uniformly bounded by our hypothesis  $\gamma_i >-1+\epsilon$, $1\le i \le q$. Therefore,
  \[
   \|f_n \otimes_{I,\psi}  f_n \|^2_{\HH^{\otimes(2q-2r)}} \le C  A_\gamma^4 \int_{[0,1]^4}
   |s_2-s_1|^{\alpha_1}
   |s_3-s_4|^{\alpha_1} |s_3-s_1|^{\alpha_2}
     |s_4-s_2|^{\alpha_3}
ds_1ds_2ds_3ds_4,
\]
where
\begin{eqnarray*}
\alpha_1&=&\sum_{i\in I}(\gamma_i+\gamma_{\psi(i)} )+r,  \\
\alpha_2&=&2\sum_{j\in I^c}\gamma_j +q-r, \\
\alpha_3&=&2\sum_{k\in J^c}\gamma_k +q-r.
\end{eqnarray*}
We have
\[
2\alpha_1 + \alpha_2 + \alpha_3 = 2 \sum_{j=1}^q +2q > -3.
\]
Therefore, from Lemma 3.3 in \cite{BaiTa2},  the above integral has a finite limit as  $\gamma_1 +\cdots + \gamma_q \rightarrow -\frac {q+1}2$, and because $A_\gamma$ converges to zero as $\gamma_1 +\cdots + \gamma_q \rightarrow -\frac {q+1}2$, this proves condition (ii). 
   \end{proof}
 
\medskip
\noindent
{\bf Remark 1.}
Functional versions of theorems Theorem \ref{thm4} and Theorem \ref{thm5} in the space $C([0,T])$ can be proved by the same arguments as in \cite{BaiTa2}. In fact, using the self-similarity and stationary-increment property of the process $Z_\gamma$, together with the  hypercontractive inequality for multiple stochastic integrals, we can show that
\[
E( | Z_{\gamma}(t) - Z_{\gamma}(s)|^p) \le c_p |t-s|^{pH },
\]
for any $p\ge 2$, where $H=\gamma_1 +\cdots + \gamma_q +q \ge \frac 12$. This leads to the tightness property and the convergence of the finite dimensional distributions is also easy to obtain.

  \medskip
\noindent
{\bf Remark 2.}  We can derive the rate of convergence in Theorem \ref{thm4} using the inequality (\ref{rate}). More precisely, it is not difficult to show that
\[
\sup_{\varphi\in \mathcal{C}^3, \|\varphi'''\|_\infty \le1, \|\varphi ''\|_\infty \le 1}  |E[\varphi(F_n)] - E[\varphi(S\eta )]| \le C \sqrt{-1-2\gamma_1}.
\]
The same rate was obtained when $q=2$ for the Wasserstein distance in \cite[Theorem 5.3]{BaiTa2}, using  properties of the second order chaos.
Concerning Theorem \ref{thm5}, using  Stein's method and the optimal rate of convergence  in the Fourth Moment Theorem derived by Nourdin and Peccati in \cite{NuPe2},  we can obtain the following rate of convergence  for the total variation distance,  as in \cite[Theorem 5.1]{BaiTa2}:
\[
 c_1 \left(\bar{\gamma}+ (q+1)/2\right)^{\frac 32} \le d_{TV} (Z_\gamma, \eta)  \le c_2 \left(\bar{\gamma}+ (q+1)/2\right)^{\frac 32},
 \]
 where $\eta$ is a $N(0,1)$ random variable. In this inequaliy $\gamma $ satisfies $\gamma_i > -1+\epsilon$, $1\le i\le q$  and the distance of $\gamma$ to  the boundary $\{ \bar{\gamma} + \frac{q+1} 2=0\}  $ is less than $\epsilon$, for some $\epsilon>0$.
To show these inequalities we need to estimate $E[ Z_\gamma ^3]$ using again the product formula for multiple stochastic integrals. We omit the details of this proof.


\begin{thebibliography}{99}
 
\bibitem{BaiTa1} S. Bai and M. Taqqu: Structure of the third moment of the generalized Rosenblatt distribution. {\it Stoch. Proc. Appl.} {\bf 124} no. 4, 1710-1739, 2014.
\bibitem{BaiTa2} S. Bai and M. Taqqu: Behavior of the generalized Rosenblatt process at extreme critical exponent values.  {\it Ann. Probab.} 

\bibitem{NoNu}    I. Nourdin and D. Nualart. Central limit theorems for multiple Skorohod integrals. {\it J. Theoret. Probab.}  {\bf 23}, no. 1, 39-64, 2010.

\bibitem{NoNuPe} I. Nourdin, D. Nualart and G. Peccati: Quantitative stable limit theorems on the Wiener space. {\it Ann.  Probab. }{\bf 44} no 1, 1-41, 2016.
\bibitem{np-book} I. Nourdin and G. Peccati. {\it Normal Approximations with Malliavin Calculus. From Stein's Method to Universality}. Cambridge University Press. 2012.

\bibitem{NuPe2} I. Nourdin and G. Peccati: The optimal fourth moment theorem. {\it Proc. Amer. Math. Soc.} {\bf 143}, 3123-3133, 2015.
\bibitem{Nu} D. Nualart. {\it The Malliavin calculus and related topics.}
 Springer-Verlag, Berlin, second edition. 2006

\bibitem{NuPe} D. Nualart and G. Peccati. Central limit theorems for sequences of multiple stochastic integrals.
{\it Ann. Probab.} {\bf 33}, no. 1, 177-193, 2005.

\bibitem{MaTu} M. Maejima and C. A. Tudor: Selfsimilar processes with stationary increments in the second Wiener chaos. {\it  Probability and Mathematical Statistics} {\bf 32}, no. 1, 167-186, 2012.

\bibitem{Sk}  A. V. Skorohod: On a generalization of a stochastic integral. {\it Theory Probab. Appl.} {\bf 20}, 219-233, 1975.


\bibitem{Ta} M. Taqqu: Weak convergence to fractional Brownian motion and to the Rosenblatt process. {\it Probab. Theory Rel. Fields} {\bf 31} no. 4, 287-302, 1975.
\end{thebibliography}
\end{document}